\RequirePackage{rotating} 
    \documentclass{amsart}

    \usepackage[utf8]{inputenc}
    \usepackage{verbatim}
    \usepackage{amsfonts}
    \usepackage{amsmath}
    \usepackage{amsthm}
    \usepackage{amssymb} 
    \usepackage{mathtools} 
    \usepackage{stmaryrd} 
    \usepackage{bm} 
    \usepackage{enumitem} 
    \usepackage{rotating} 

    \usepackage{multirow}
    \usepackage{tabularx} 
        \newcolumntype{C}{>{\centering\arraybackslash}X}
        \newcolumntype{K}[1]{>{\centering\arraybackslash}m{#1}}
    
    \usepackage[maxbibnames=99]{biblatex}
        \DeclareFieldFormat{title}{\mkbibquote{#1}}
    
    \usepackage[hidelinks]{hyperref}
    \usepackage[noabbrev,nameinlink,capitalise]{cleveref}
        \crefname{subsection}{Subsection}{Subsections}
        \crefname{subsection}{Subsection}{Subsections}
    
    \usepackage{subcaption}
    \usepackage{color}
    \usepackage[dvipsnames]{xcolor}
    \usepackage{graphicx}\graphicspath{{./images/}}
    \usepackage{tikz-cd} 
    \usepackage{tikz}
        \usetikzlibrary{arrows,arrows.meta,shapes}
        \tikzset{every picture/.style=thick}
        \tikzset{state/.style = {shape=ellipse,draw,inner sep=1pt, outer sep=1pt}}
        \tikzset{vertex/.style = {shape=circle,draw,inner sep=1pt, outer sep=1pt}}
        \tikzset{edge/.style = {->, >={latex[scale=1.5]}}}

    \theoremstyle{plain}
        \newtheorem{theorem}{Theorem}[section]
        
        \newtheorem{lemma}[theorem]{Lemma}
        
    \theoremstyle{definition}
        \newtheorem{definition}[theorem]{Definition}
        
        \newtheorem{assumption}[theorem]{Assumption}
    \theoremstyle{remark}
        \newtheorem{remark}[theorem]{Remark}

    \title{Rational Gluing in Edge Replacement Systems}
    \author{Davide Perego \and Matteo Tarocchi}
    \date{}
    
    \thanks{Both the authors are members of the Gruppo Nazionale per le Strutture Algebriche, Geometriche e le loro Applicazioni (GNSAGA) of the Istituto Nazionale di Alta Matematica (INdAM).
    The first author acknowledges support from the Grant QUALIFICA by Junta de Andalucía grant number QUAL21 005 USE and from the research grant PID2022-138719NA-I00 (Proyectos de Generación de Conocimiento 2022) financed by
    the Spanish Ministry of Science and Innovation.}

    \address{Instituto de Matemáticas, Universidad de Sevilla, Ed. Celestino Mutis, Campus de Reina Mercedes, Avda. Reina Mercedes s/n, 41012 Sevilla, Spain, EU}
    \email{\href{mailto:dperego9@gmail.com}{dperego9@gmail.com}}
    \address{Dipartimento di Matematica e Applicazioni, Universit\`a degli Studi di Milano-Bicocca, Ed. U5, Via R.Cozzi 55, 20125 Milano, Italy, EU}
    \email{\href{mailto:matteo.tarocchi.math@gmail.com}{matteo.tarocchi.math@gmail.com}}

\begin{document}

\begin{abstract}
In this paper, we prove the rationality of the gluing relation of edge replacement systems, which were introduced for studying rearrangement groups of fractals.
More precisely, we describe an algorithmic procedure for building a finite state automaton that recognizes pairs of equivalent sequences that are glued in the fractal.
This fits in recent interest towards the rationality of gluing relations on totally disconnected compact metrizable spaces.
\end{abstract}



\maketitle


\section*{Introduction}

In this paper we study the limit space of edge replacement systems, which arose from the context of rearrangement groups of fractals introduced in \cite{BF19}.
This is a family of Thompson-like groups acting on fractals which have now been studied and used in several papers \cite{MR4033504,airplane,IG,BasilicaDense,RearrConj,Quasisymmetries,dendrite}.
We will construct automata that recognize pairs of infinite words precisely when they are equivalent under the gluing relation that defines the limit space, meaning that the gluing relation is \textit{rational}.
Different ways to define rationality have been employed in the literature, especially in symbolic dynamics with notions like edge and sofic shifts.
By the well-known Alexandroff-Hausdorff Theorem, every compact metrizable space $X$ is a quotient of the Cantor space $\mathfrak{C}$.
When seeing $\mathfrak{C}$ as a set of infinite words on a finite alphabet, this defines a \textit{coding} on $X$:
every point of $X$ is associated to an equivalence class of infinite words of $\mathfrak{C}$ that are ``glued'' by the equivalence relation.
More generally, the Cantor space $\mathfrak{C}$ can replaced by a one-sided edge shift, which is the set of infinite words that can be read as walks on a finite directed graph.
A natural and simple example of a coding is the binary coding of the unit interval, where each number in $[0,1]$ has either one or two binary codings depending on whether it is dyadic.
While a general coding can be quite unwieldy, interesting and useful codings generally arise when there is an explicit rule for gluing two infinite words together.
What motivates the interest in the \textit{rationality} of a gluing relation is thus that there is indeed a rule that describes when two words glue together, and that rule is ``computable''.

There are several examples of this in literature, though they may be addressed with different names or be stronger or weaker versions of what we are describing here.
The Gromov boundary of a hyperbolic group has rational gluing with respect to different codings.
In \cite{CP,P}, it is proved that it is \textit{semi-Markovian} with respect to the boundary of the tree of geodesics and other trees introduced by Gromov himself (\cite{G}).
Being semi-Markovian is generally a stronger notion than rationality (see again \cite{P}).
In \cite{Perego}, it is shown that the gluing associated to the \textit{horofunction boundary} is rational. Limit spaces of contracting self similar groups have rational gluing with respect to the regular tree on which the groups act (see \cite{N2}).
In \cite{Bandt}, Bandt explores the opposite point of view, building self-similar spaces starting from automata that define gluing relations and our construction fits in Bandt's formalism.

Recently in \cite{BF19}, for their study in group theory, Belk and Forrest introduced \textit{edge replacement systems}, 
which are graph rewriting systems that produce a plethora of different self-similar fractals, such as the basilica \cite{BF15} and airplane Julia sets \cite{airplane} and Wa\.zewski dendrites \cite{dendrite}.
As we will recall in the background Section, these fractals are obtained as quotients of the Cantor space under an equivalence ``gluing'' relation.
Here we prove that this relation is rational,
as conjectured in \cite[Example 7.14]{Perego}.
We provide an explicit construction that associates to any edge replacement system a finite state automaton that recognizes pairs of equivalent infinite words.


\section{Background on Edge Replacement Systems}

The definitions given in this Section are the same of \cite{BF19}, except that we will associate distinguished colors to loops as explained in \cref{sub.loops} (this change is only technical and does not affect the limit spaces that can be achieved).
For a complete treatment of this topic we refer to \cite{BF19}.

\subsection{Edge Replacement Systems}
\label{sub:rep:sys}

In the remainder of this paper, a \textbf{graph} will be a sixtuplet $(V, E, \mathrm{Col}, \iota, \tau, \mathrm{c})$ where:
\begin{itemize}
    \item $V$, $E$ and $\mathrm{Col}$ are finite sets of \textbf{vertices}, \textbf{edges} and \textbf{colors}, respectively;
    \item $\iota$ and $\tau$ are maps $E \to V$, whose images $\iota(e)$ and $\tau(e)$ are called the \textbf{initial vertex} and the \textbf{terminal vertex} of the edge $e$;
    \item $\mathrm{c}$ is a map $E \to \mathrm{Col}$, and each image $\mathrm{c}(e)$ is called the \textbf{color} of the edge $e$.
\end{itemize}
When the set of colors is a singleton, we will denote by $1$ the only color and we will essentially omit $\mathrm{Col}$, $\mathrm{c}$ and any related index from the notation.
Moreover, given a graph $\Gamma$ it will often be useful to denote the set of edges of $\Gamma$ by $\Gamma^1$.

Note that we allow parallel edges (i.e., edges $e_1 \neq e_2$ such that $\iota(e_1) = \iota(e_2)$ and $\tau(e_1) = \tau(e_2)$) and loops (i.e., edges $e$ such that $\iota(e) = \tau(e)$).

\begin{definition}
\label{def:rep:sys}
A \textbf{replacement system} is a tuple $(\mathrm{Col}, \Gamma_0, \{\Gamma_{\mathrm{c}}\}_{\mathrm{c} \in \mathrm{Col}})$ where:
\begin{itemize}
    \item $\mathrm{Col}$ is a finite set of \textbf{colors};
    \item $\Gamma_0$ is a finite graph, called the \textbf{base graph};
    \item for each $\mathrm{c} \in \mathrm{Col}$, $\Gamma_{\mathrm{c}}$ is a finite graph, called the \textbf{$\mathrm{c}$ replacement graph};
    \item each replacement graph $\Gamma_{\mathrm{c}}$ is equipped with two distinct \textbf{boundary vertices} $\iota_{\mathrm{c}}$ and $\tau_{\mathrm{c}}$.
\end{itemize}
\end{definition}

Examples of replacement systems are depicted in \cref{fig:replacement:interval,,fig:replacement:dendrite}.

\begin{figure}
\centering
\begin{tikzpicture}
    \node at (-.667,0) {$\Gamma_0 =$};
    \node[vertex] (i) at (0,0) {};
    \node[vertex] (t) at (2,0) {};
    \draw[edge] (i) to node[above]{$S$} (t);
    \begin{scope}[xshift=6.5cm]
    \node at (-2,0) {$\Gamma_1 =$};
    \node[vertex] (l) at (180:1.3) {};
    \draw (l) node[above]{$\iota$};
    \node[vertex] (c) at (0:0) {};
    \node[vertex] (r) at (0:1.333) {};
    \draw (r) node[above]{$\tau$};
    \draw[edge] (l) to node[above]{$0$} (c);
    \draw[edge] (c) to node[above]{$1$} (r);
    \end{scope}
\end{tikzpicture}
\caption{The interval replacement system $\mathcal{I}$.}
\label{fig:replacement:interval}
\end{figure}
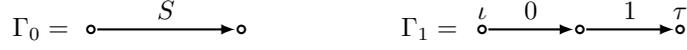

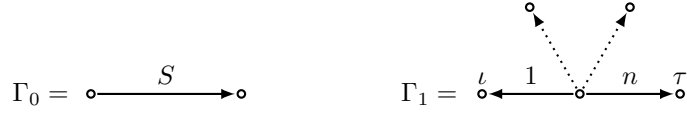
\begin{figure}
\centering
\begin{tikzpicture}
    \node at (-.667,0) {$\Gamma_0 =$};
    \node[vertex] (i) at (0,0) {};
    \node[vertex] (t) at (2,0) {};
    \draw[edge] (i) to node[above]{$S$} (t);
    \begin{scope}[xshift=6.5cm]
    \node at (-2,0) {$\Gamma_1 =$};
    \node[vertex] (l) at (180:1.3) {};
    \draw (l) node[above]{$\iota$};
    \node[vertex] (c) at (0:0) {};
    \node[vertex] (tl) at (60:1.333) {};
    \node[vertex] (tr) at (120:1.333) {};
    \node[vertex] (r) at (0:1.333) {};
    \draw (r) node[above]{$\tau$};
    \draw[edge] (c) to node[above]{$1$} (l);
    \draw[edge,dotted] (c) to (tl);
    \draw[edge,dotted] (c) to (tr);
    \draw[edge] (c) to node[above]{$n$} (r);
    \end{scope}
\end{tikzpicture}
\caption{The Dendrite replacement systems $\mathcal{D}_n$.}
\label{fig:replacement:dendrite}
\end{figure}

\subsubsection{Colors and Loops}
\label{sub.loops}

Throughout this paper, we will make the following assumption on the graph expansions of our edge replacement systems, as it will make it easier to build our gluing automaton.

\begin{assumption}
\label{ass.loops}
We assume that edges of the same color are either all loops or all not loops.
\end{assumption}

Under this assumption, it is natural to identify the two boundary vertices for colors that represent loops.
Each replacement graph $\Gamma_c$ will then equipped with either one or two \textbf{boundary vertices}:
a single vertex $\lambda_c$ for loops and two distinct vertices $\iota_c$ and $\tau_c$ for non-loops.

\begin{remark}
This assumption is not really a restriction on the limit spaces that we can produce, nor on their rearrangement groups.
Indeed, given a generic edge replacement system, we can always split a color $c$ into two colors $c_\text{loop}$ and $c_\text{non-loop}$ and create a new edge replacement system that produces the same limit space and the same rearrangement group.
This creates a replacement system that satisfies \cref{ass.loops} because whether an edge is a loop or not only depends on whether it is a loop of the base graph or of some replacement graph.
For example, the original basilica replacement system \cref{fig:basilica:replacement:non:loops} becomes the one portrayed in \cref{fig:replacement:basilica}.

More precisely, the new edge replacement system $\mathcal{R}^*$ is obtained
by replacing each color $c$ by two new colors $c_\text{loop}$ and $c_\text{non-loop}$, where $c_\text{loop}$ is used on loops that were formerly of color $c$ and $c_\text{non-loop}$ is used on every other edge of color $c$.
Even if we will not discuss rearrangement groups here, it is also not hard to see that this modification does not affect the rearrangement group.
\end{remark}

\begin{figure}
\centering
\begin{tikzpicture}
    \node at (-.8,0) {$\Gamma_0 =$};
    \node[vertex] (i) at (1,0) {};
    \draw[edge] (i) to[out=150, in=-150, min distance=1.2cm] node[above left]{$L$} (i);
    \draw[edge] (i) to[out=-30, in=30, min distance=1.2cm] node[above right]{$R$} (i);
    \begin{scope}[xshift=4.25cm]
    \node at (-.825,0) {$\Gamma_{1} =$};
    \node[vertex] (l) at (0,0) {};
    \draw (l) node[above]{$\iota_{1}$};
    \node[vertex] (c) at (1,0) {};
    \node[vertex] (r) at (2,0) {};
    \draw (r) node[above]{$\tau_{1}$};
    \draw[edge] (l) to node[above]{$1$} (c);
    \draw[edge] (c) to[out=60, in=120, min distance=1.2cm] node[above right]{$2$} (c);
    \draw[edge] (c) to node[above]{$3$} (r);
    \end{scope}
\end{tikzpicture}
\caption{The original Basilica replacement system.}
\label{fig:basilica:replacement:non:loops}
\end{figure}
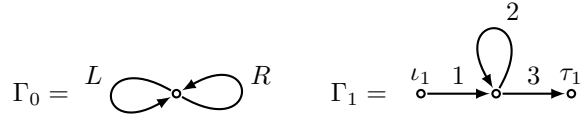

\begin{figure}
\centering
\begin{tikzpicture}
    \node at (-.8,0) {$\Gamma_0 =$};
    \node[vertex] (i) at (1,0) {};
    \draw[edge,red] (i) to[out=150, in=-150, min distance=1.2cm] node[above left]{$L$} (i);
    \draw[edge,red] (i) to[out=-30, in=30, min distance=1.2cm] node[above right]{$R$} (i);
    \begin{scope}[xshift=4.25cm]
    \node at (-.825,0) {$\Gamma_{\textcolor{blue}{1}} =$};
    \node[vertex] (l) at (0,0) {};
    \draw (l) node[above]{$\iota_{\textcolor{blue}{1}}$};
    \node[vertex] (c) at (1,0) {};
    \node[vertex] (r) at (2,0) {};
    \draw (r) node[above]{$\tau_{\textcolor{blue}{1}}$};
    \draw[edge,blue] (l) to node[above]{$1$} (c);
    \draw[edge,red] (c) to[out=60, in=120, min distance=1.2cm] node[above right]{$2$} (c);
    \draw[edge,blue] (c) to node[above]{$3$} (r);
    \end{scope}
    \begin{scope}[xshift=8.5cm]
    \node at (-1,0) {$\Gamma_{\textcolor{red}{2}} =$};
    \node[vertex] (l) at (0,0) {};
    \draw (l) node[left]{$\lambda_{\textcolor{red}{2}}$};
    \node[vertex] (c) at (1,0) {};
    \draw[edge,blue] (l) to[out=290, in=240] node[below]{$4$} (c);
    \draw[edge,red] (c) to[out=-30, in=30, min distance=1.2cm] node[above right]{$5$} (c);
    \draw[edge,blue] (c) to[out=110,in=60] node[above]{$6$} (l);
    \end{scope}
\end{tikzpicture}
\caption{The modified Basilica replacement system $\mathcal{B}$.}
\label{fig:replacement:basilica}
\end{figure}
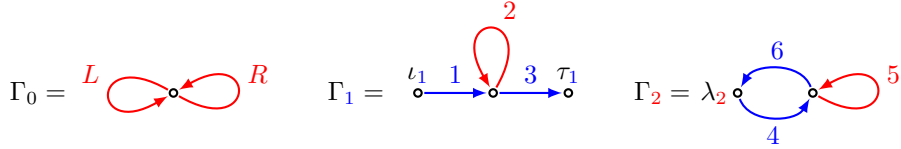

\subsubsection{Graph Expansions}

A replacement system defines a rewriting system in the following way.
We can \textbf{expand} the base graph $\Gamma_0$ by replacing one of its edges $e$ with the $\mathrm{c}(e)$-colored replacement graph $\Gamma_{\mathrm{c}(e)}$, attaching the boundary vertices of $\Gamma_{\mathrm{c}(e)}$ in place of the initial and terminal vertices $\iota_{\mathrm{c}}(e)$ and $\tau_{\mathrm{c}}(e)$ of the edge $e$.
Note that, because we identify the boundary vertices for colors associated to loops as explained in \cref{sub.loops}, this is coherent with the amount of vertices that need to be attached (one for loops, two otherwise).
This very same procedure can be applied to any graph expansion of the base graph, which allows us to obtain further graph expansions.
For example, \cref{fig:expansions} depicts graph expansions of the replacement systems $\mathcal{I}$ and $\mathcal{B}$ from \cref{fig:replacement:basilica,,fig:replacement:interval}, respectively.
This prompts the following useful definition.

\begin{definition}
\label{def.full.expansion}
The \textbf{full expansion sequence} is the sequence of graphs $E_m$ obtained by expanding, at each step, every edge of the previous graph, starting with the base graph $E_0 = \Gamma_0$.
\end{definition}

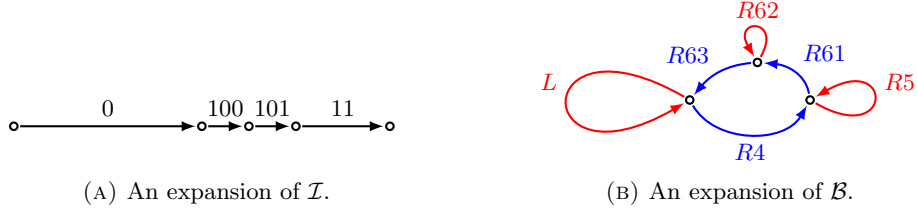
\begin{figure}
    \centering
    \begin{subfigure}[B]{.45\textwidth}
        \centering
        \begin{tikzpicture}[scale=5]
            \node[vertex] (0) at (0,0) {};
            \node[vertex] (1/2) at (.5,0) {};
            \node[vertex] (5/8) at (.625,0) {};
            \node[vertex] (3/4) at (.75,0) {};
            \node[vertex] (1) at (1,0) {};
            \draw[edge] (0) to node[above]{\small$0$} (1/2);
            \draw[edge] (1/2) to node[above]{\small$100$} (5/8);
            \draw[edge] (5/8) to node[above]{\small$101$} (3/4);
            \draw[edge] (3/4) to node[above]{\small$11$} (1);
        \end{tikzpicture}
        \vspace{12pt}
        \caption{An expansion of $\mathcal{I}$.}
    \end{subfigure}
    \hfill
    \begin{subfigure}[B]{.45\textwidth}
        \centering
        \begin{tikzpicture}[scale=2]
            \node[vertex] (i) at (1,0) {};
            \node[vertex] (r) at (1.8,0) {};
            \node[vertex] (rt) at (1.45,.25) {};
            \draw[edge,red] (i) to[out=150, in=-150, min distance=1.2cm] node[above left]{\small$L$} (i);
            \draw[edge,blue] (i) to[out=300, in=240] node[below]{\small$R4$} (r);
            \draw[edge,red] (r) to[out=-30, in=30, min distance=.6cm] node[above right]{\small$R5$} (r);
            \draw[edge,blue] (r) to[out=100,in=-5] node[above right]{\small$R61$} (rt);
            \draw[edge,red] (rt) to[out=60, in=120, min distance=.3cm] node[above]{\small$R62$} (rt);
            \draw[edge,blue] (rt) to[out=185,in=50] node[above left]{\small$R63$} (i);
        \end{tikzpicture}
        \vspace{-12pt}
        \caption{An expansion of $\mathcal{B}$.}
        \label{fig:expansion:B}
    \end{subfigure}
    \caption{Examples of graph expansions.}
    \label{fig:expansions}
\end{figure}

\subsection{Limit Spaces}
\label{sub:lim:spc}

In this Section we identify the edges of graph expansions with certain finite words, then we consider infinite words and define the limit space as a quotient under a ``gluing'' relation.
For the sake of clarity, from here on we will use the term \textit{word} to refer to finite words and the term \textit{sequence} for infinite words.

\subsubsection{The Color Graph}

Let $\mathcal{R} = (\mathrm{Col}, \Gamma_0, \{\Gamma_{\mathrm{c}}\}_{\mathrm{c} \in \mathrm{Col}})$ be a fixed replacement system.
When performing expansions, each edge is represented by some finite word in the alphabet consisting of all edges of the base and replacement graphs.
For instance, in the example portrayed in \cref{fig:expansion:B} the word $R4$ means that we have expanded the edge $R$ of the base graph and then we consider the edge corresponding to $4$ among the newly added edges.
More precisely, the set of those words that correspond to edges of graph expansions is denoted by $E_\mathcal{R}$, and it can be described with the aid of the following definition.

\begin{definition}
\label{def:color:automaton}
The \textbf{color graph} $A_\mathcal{R}$ of the replacement system $\mathcal{R}$ is the directed (non-colored) graph consisting of:
\begin{itemize}
    \item a vertex $q(\mathrm{c})$ for each color $\mathrm{c} \in \mathrm{Col}$ and an additional vertex $q(0)$;
    \item an edge $q(\zeta) \rightarrow q(\mathrm{c})$ for each $\mathrm{c}$-colored edge of $\Gamma_\zeta$, where $\zeta \in \{0\} \cup \mathrm{Col}$.
\end{itemize}
\end{definition}

Note that, from every vertex $q(\mathrm{c_1})$, the color graph $A_\mathcal{R}$ has an outgoing edge for each edge $e$ of the replacement graph $\Gamma_{\mathrm{c_1}}$. 
Each of these edges of $A_\mathcal{R}$ ends at the vertex $q(\mathrm{c_2})$, where $\mathrm{c_2}$ is the color of $e$.
With this in mind, it is clear that the set of edges $E_\mathcal{R}$ corresponds to the set of finite directed walks on $A_\mathcal{R}$ starting from $q(0)$.

\subsubsection{The Limit Space}

Since longer words in $E_\mathcal{R}$ correspond to ``smaller'' edges of the graph expansions, it is natural to consider the infinite sequences of these words.

\begin{definition}
Given a replacement system $\mathcal{R}$, its \textbf{symbol space} $\Omega_\mathcal{R}$ is the (one-sided) edge shift on $A_\mathcal{R}$ starting at $q(0)$, i.e., $\Omega_\mathcal{R}$ is the set of infinite directed walks on the color graph $A_\mathcal{R}$ that start from the vertex $q(0)$.
\end{definition}

We endow $\Omega_\mathcal{R}$ with the product topology.
Under the hypothesis of an expanding replacement system (\cref{def:exp:rep:sys} below, which will be assumed for the remainder of this paper), $\Omega_\mathcal{R}$ is always going to be a Cantor space with this topology.

The following definition translates the adjacencies of edges in the base and replacement graphs into a relation between sequences of the symbol space $\Omega_\mathcal{R}$.

\begin{definition}
\label{def:gluing}
Two sequences $\alpha$ and $\beta$ of $\Omega_\mathcal{R}$ are equivalent under the \textbf{gluing relation} if $\alpha_1 \dots \alpha_n$ and $\beta_1 \dots \beta_n$ are incident on a common vertex for every large enough $n \in \mathbb{N}$.
In this case, we write $\alpha \sim \beta$.
\end{definition}

This relation $\sim$ is supposed to glue together those sequences that represent the same point in the limit space.
If two distinct sequences $\alpha$ and $\beta$ of $\Omega_\mathcal{R}$ are glued by $\sim$, then they must share a (possibly trivial) common prefix $x$, after which $x_1 \dots x_k \alpha_{k+1} \dots \alpha_{m}$ and $x_1 \dots x_k \beta_{k+1} \dots \beta_{m}$ must be incident on a common vertex for all $m \geq k$.
For example, in the interval replacement system $\mathcal{I}$, for every finite prefix $p$ one has $p 0 \bar{1} \sim p 1 \bar{0}$, 
and these are all the gluings on $\mathcal{I}$.

The gluing relation is not always an equivalence relation, but it is by \cite[Proposition 1.9]{BF19} when the replacement system satisfies the following conditions.

\begin{definition}
\label{def:exp:rep:sys}
A replacement system $\mathcal{R}$ is \textbf{expanding} if:
\begin{enumerate}
    \item the base and replacement graphs have no isolated vertices;
    \item there is no edge between $\iota_c$ and $\tau_c$ in any of the replacement graphs;
    \item each replacement graph has at least two edges and at least one vertex that is not a boundary vertex;
\end{enumerate}
\end{definition}

This condition needs to be checked on the original replacement system, \textit{before} applying the modification described in \cref{sub.loops} that guarantees \cref{ass.loops}.

\begin{definition}
The \textbf{limit space} of an expanding replacement system $\mathcal{R}$ is the quotient $X_\mathcal{R} = \Omega_\mathcal{R} / \sim$.
\end{definition}

By \cite[Theorem 1.25]{BF19}, limit spaces are non-empty compact metrizable spaces.
The limit space for the replacement system $\mathcal{I}$ from \cref{fig:replacement:interval} is the (topological) interval.
The limit space of $\mathcal{D}_n$ from \cref{fig:replacement:dendrite} is the Wa\.zewski dendrite $D_n$ of order $n$, and \cref{fig:dendrite}\footnote{\href{https://commons.wikimedia.org/wiki/File:Julia_IIM_1.jpg}{image} in \cref{fig:dendrite} by Adam Majewski, licensed under \href{https://creativecommons.org/licenses/by/3.0}{CC BY 3.0}} depicts $D_3$;
see \cite{dendrite} for more details on these spaces and their rearrangement groups.
The limit space of $\mathcal{B}$ from \cref{fig:replacement:basilica} (or, equivalently, \cref{fig:basilica:replacement:non:loops}) is the Basilica limit space, which is homeomorphic to the Julia set for the quadratic polynomial map $z^2 - 1$, depicted in \cref{fig:basilica}\footnote{image in \cref{fig:basilica} by Jim Belk and Bradley Forrest from \cite{BF15}};
see \cite{BF15,BF19} for more details on the limit space and its rearrangement group.

\begin{figure}
    \centering
    \begin{subfigure}[B]{.45\textwidth}
        \centering
        \includegraphics[width=\textwidth]{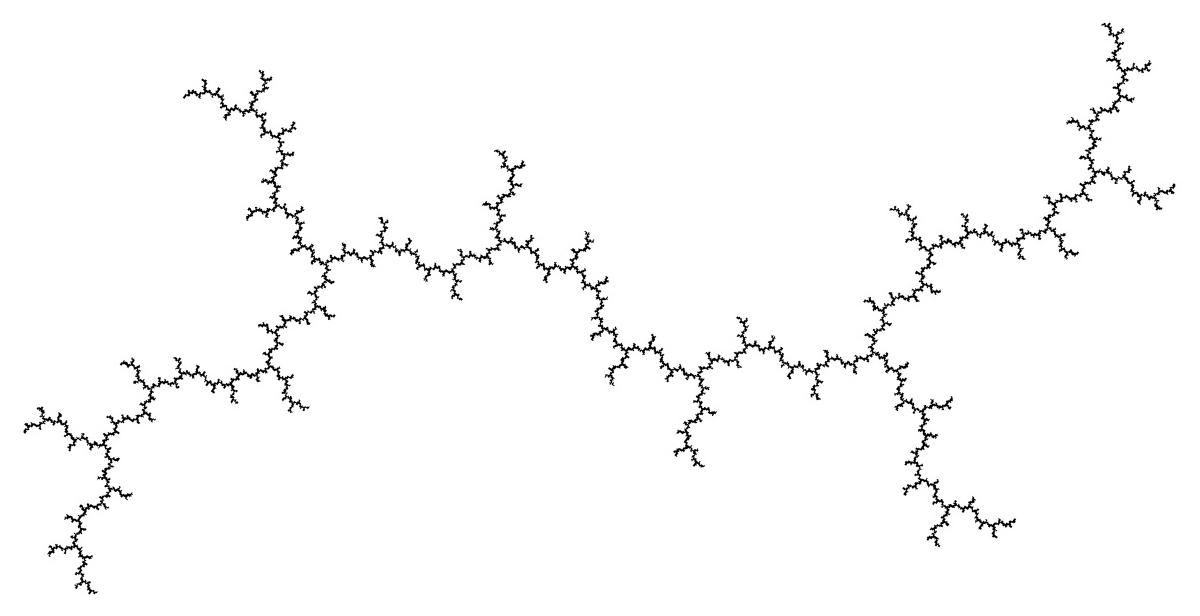}
        \caption{The Wa\.zewski dendrite $D_3$.}
        \label{fig:dendrite}
    \end{subfigure}
    \hfill
    \begin{subfigure}[B]{.45\textwidth}
        \centering
        \includegraphics[width=\textwidth]{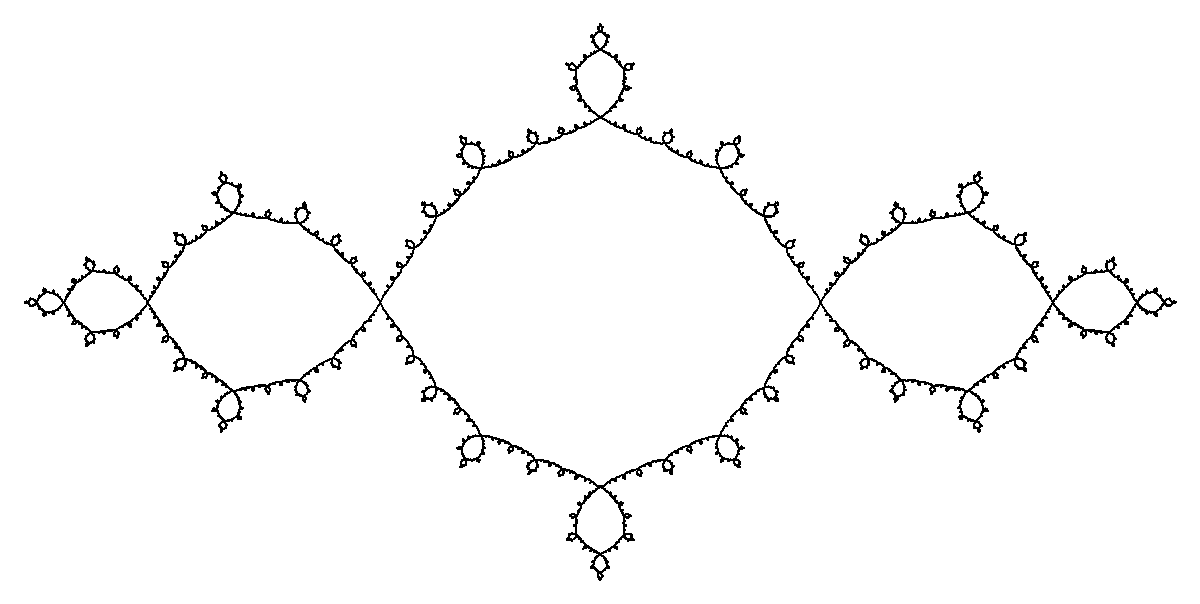}
        \caption{The Basilica.}
        \label{fig:basilica}
    \end{subfigure}
    \caption{Example of limit spaces.}
\end{figure}


\section{The Gluing Automaton}

In this section, after briefly recalling what an automaton is and how it can be used to describe a relation between infinite words, we construct the gluing automaton for edge replacement systems.

A \textit{partial function} is a binary relation between two sets that associates to every element of the first set at most one element of the second. 

\begin{definition}
Let $\Sigma$ be a finite set called \textit{alphabet}.
A (synchronous deterministic finite state) \textbf{automaton} is a quadruple $(\Sigma, Q, \rightarrow, q_{0})$ with $Q$ a finite set whose elements are called \textit{states}, $q_{0} \in Q$ called \textit{initial state} and a partial function between $Q \times \Sigma $ and $Q$ which is called \textit{transition function}.
\end{definition}

We write $q \xrightarrow{x} q'$ if the transition function features $(q,x) \mapsto q'$.
A \textbf{run} of the automaton is a sequence of transitions $q_0 \xrightarrow{x_1} q_1 \xrightarrow{x_2} \dots$ and an automaton \textbf{recognizes} a sequence $x_1 x_2 \dots$ when there exists such a run.
With a slight abuse of notation, the same concept naturally applies to finite words too.
An automaton \textbf{recognizes} a set $\Omega$ when it recognizes a sequence if and only if it belongs to $\Omega$.

\begin{definition}
Let $\Omega$ be a set of infinite sequences.
We say that an equivalence relation on $\Omega$ is \textbf{rational} when there exists an automaton that recognizes a pair $(w, v)$ if and only if the two elements $w, v \in \Omega$ are related.
\end{definition}

Note that an automaton that recognizes a relation on $\Sigma^\omega = \{ x_1 x_2 \ldots \mid x_i \in \Sigma \}$
has $\Sigma^2$ as alphabet.
A sequence $(x_1,y_1) (x_2,y_2) \dots$ represents a pair of sequences $x_1 x_2 \ldots, y_1 y_2 \ldots \in \Omega_\mathcal{R}$.

\subsection*{Construction of the Gluing Automaton}

Now consider an expanding replacement system $\mathcal{R}$ with colors $\mathrm{Col} = \{1, \dots, k\}$, base graph $\Gamma_0$ and replacement graphs $\Gamma_1, \dots, \Gamma_k$.
If the replacement system has loops, from here on we will always consider the replacement system where the boundary vertices for colors associated to loops are identified, as explained in \cref{sub.loops}.
The gluing automaton $\mathrm{Gl}_\mathcal{R} = (\Sigma^2, Q, \rightarrow, q_0(0))$ is built as follows.

\subsubsection*{The alphabet}

$\Sigma^2$ consists of all possible pairs of symbols of $\Sigma$, where $\Sigma$ has exactly one symbol for each edge of the base and replacement graphs, i.e.,
\[ \Sigma^2 = \{ (x,y) \mid x,y \in \cup_{i = 0, \dots, k} \Gamma_i^1 \}. \]
When processing one of the two entries of the $l$-th letter of a word in this alphabet, it will be natural to refer to that entry as if it were an edge of the $l$-th graph of the full expansion sequence (\cref{def.full.expansion}).
In this way, the two entries of the $l$-th letter of a word in $\Sigma^2$ represent two edges of the same graph, so in what follows we will simply describe their possible adjacency in the $l$-th graph of the full expansion sequence without explicitly mentioning the graph itself.

\subsubsection*{The set of states}

$Q$ is the union of the two following sets $Q_0$ and $Q_1$:
\[ Q_0 = \big\{ q_0(i) \mid i = 0, 1, \dots, k \big\}, \]
\[ Q_1 = \big\{ q_1 \big(\begin{smallmatrix} i & \gamma \\ j & \delta \end{smallmatrix}\big) \mid i, j \in \mathrm{Col}, \gamma, \delta \in \mathcal{E}, \gamma \neq \mathrm{db}^-, \delta=\mathrm{db}^\pm \iff \gamma=\mathrm{db}^+ \big\}, \]
where $\mathcal{E} = \{\mathrm{in}, \mathrm{out}, \mathrm{lp}, \mathrm{db}^+, \mathrm{db}^- \}$.
The initial state is $q_0(0) \in Q_0$.

The following Subsection will describe in more detail what these states mean.
The intuition is that, while processing the edges of our two sequences, the first set $Q_0$ consists of an initial state $q_0(0)$ along with a state $q_0(i)$ for each color $i$ that indicates that the two entries being processed are the same $i$-colored edge, while the second set $Q_1$ has a state for each quadruple $\big(\begin{smallmatrix} i & \gamma \\ j & \delta \end{smallmatrix}\big)$ where $i, j$ are the colors of the edges and $\gamma, \delta$ distinguish between the types of adjacency between the edges.
The symbols $\mathrm{lp}$ and $\mathrm{db}$ stand for \textit{loop} and \textit{double}, respectively, and their meaning will be made clear when describing the transition functions.

\subsubsection*{The transitions}

The transition function $\rightarrow$ is the partial mapping $Q \times \Sigma \to Q$ defined as follows, where we distinguish between three different types of transitions.
\begin{description}
    \item[Type $\bm{0}$]
        $q_0(i) \xrightarrow{(e,e)} q_0 (\mathrm{col}(e))$ for all $i \in \{0\} \cup \mathrm{Col}$ and $e \in \Gamma_i^1$.
    \item[Type $\bm{1}$]
        $q_0(i) \xrightarrow{(a,b)} q_1 \big(\begin{smallmatrix} \mathrm{col}(a) & \alpha \\ \mathrm{col}(b) & \beta \end{smallmatrix}\big)$ for all $i \in \{0\} \cup \mathrm{Col}$ and $a \neq b$ in $\Gamma_i^1$, where $\alpha$ and $\beta$ depend on the adjacency between $a$ and $b$ in $\Gamma_i$ as shown in \cref{tab:transition:type:1}.
    \item[Type $\bm{2}$]
        $q_1 \big(\begin{smallmatrix} i & \gamma \\ j & \delta \end{smallmatrix}\big) \xrightarrow{(a,b)} q_1 \big(\begin{smallmatrix} \mathrm{col}(a) & \alpha \\ \mathrm{col}(b) & \beta \end{smallmatrix}\big)$ for all $i,j \in \{0\} \cup \mathrm{Col}$ whenever $a \in \Gamma_i^1$ and $b \in \Gamma_j^1$ satisfy the conditions described in \cref{tab:transition:type:2}, which also depend on $\gamma$ and $\delta$ and determine $\alpha$ and $\beta$.
\end{description}

The meaning of these transitions (especially Type 2) will be made clear in \cref{sec:gl:aut}.
The idea is that Type 0 transitions simply induce a copy of the color graph (\cref{def:color:automaton}), Type 1 transitions arise when two sequences first diverge and Type 2 transitions describe the adjacency of sequences after their divergence.
Note that, when processing the pair $(a,b)$ in a state $q_0(i)$, there is a Type 1 transition if and only if $a$ and $b$ are adjacent in $\Gamma_i$.

Observe that loops at $\lambda_i$ may arise when the original replacement system (before gluing the initial and terminal vertices as explained in \cref{sub.loops}) had loops at $\iota$ or at $\tau$. Since the replacement system is expanding, this is the only possible occurrence of the column $\iota(a) = \tau(a) = \lambda_i$ (row $\iota(b) = \tau(b) = \lambda_j$).

Moreover, in Type 2 transitions $\alpha$ and $\beta$ will never be $\mathrm{db}^+$ nor $\mathrm{db}^-$, since the replacement system is expanding.
Furthermore, note that when $\gamma$ (or $\delta$) is $\mathrm{db}^+$ or $\mathrm{db}^-$ one needs to check the incidence of $a$ (or $b$) with both the initial and terminal vertices $\iota_i$ and $\tau_i$ (or $\iota_j$ and $\tau_j$) of the replacement graph.
However, since the replacement system is expanding, an edge of a replacement graph cannot be incident on both the initial and terminal vertices;
thus, at most one of these can occur, so there is no ambiguity in the table.

\begin{table}
\centering
\begin{tabularx}{\textwidth}{C|C|C|C}
    \begin{tikzpicture}
        \node[vertex] (A) at (0,0) {};
        \node[vertex] (B) at (1,0) {};
        \node[vertex] (C) at (2,0) {};
        \draw[edge] (A) to node[above]{$a$} (B);
        \draw[edge] (B) to node[above]{$b$} (C);
    \end{tikzpicture}
    &%
    \begin{tikzpicture}
        \node[vertex] (A) at (0,0) {};
        \node[vertex] (B) at (1,0) {};
        \node[vertex] (C) at (2,0) {};
        \draw[edge] (A) to node[above]{$a$} (B);
        \draw[edge] (C) to node[above]{$b$} (B);
    \end{tikzpicture}
    &%
    \begin{tikzpicture}
        \node[vertex] (A) at (0,0) {};
        \node[vertex] (B) at (1,0) {};
        \node[vertex] (C) at (2,0) {};
        \draw[edge] (B) to node[above]{$a$} (A);
        \draw[edge] (B) to node[above]{$b$} (C);
    \end{tikzpicture}
    &%
    \begin{tikzpicture}
        \node[vertex] (A) at (0,0) {};
        \node[vertex] (B) at (1,0) {};
        \node[vertex] (C) at (2,0) {};
        \draw[edge] (B) to node[above]{$a$} (A);
        \draw[edge] (C) to node[above]{$b$} (B);
    \end{tikzpicture}%
    \vspace{.15cm}
    \\
    $q_1 \big(\begin{smallmatrix} \mathrm{col}(a) & \mathrm{in} \\ \mathrm{col}(b) & \mathrm{out} \end{smallmatrix}\big)$%
    &%
    $q_1 \big(\begin{smallmatrix} \mathrm{col}(a) & \mathrm{in} \\ \mathrm{col}(b) & \mathrm{in} \end{smallmatrix}\big)$%
    &%
    $q_1 \big(\begin{smallmatrix} \mathrm{col}(a) & \mathrm{out} \\ \mathrm{col}(b) & \mathrm{out} \end{smallmatrix}\big)$%
    &%
    $q_1 \big(\begin{smallmatrix} \mathrm{col}(a) & \mathrm{out} \\ \mathrm{col}(b) & \mathrm{in} \end{smallmatrix}\big)$%
    \vspace{.2cm}
    \\\hline
    \begin{tikzpicture}
        \node[vertex] (A) at (0,0) {};
        \node[vertex] (B) at (1,0) {};
        \draw[edge] (A) to node[above]{$a$} (B);
        \draw[edge] (B) to[loop right, min distance=1cm] node[right]{$b$} (B);
    \end{tikzpicture}
    &%
    \begin{tikzpicture}
        \node[vertex] (A) at (0,0) {};
        \node[vertex] (B) at (1,0) {};
        \draw[edge] (B) to node[above]{$a$} (A);
        \draw[edge] (B) to[loop right, min distance=1cm] node[right]{$b$} (B);
    \end{tikzpicture}
    &%
    \begin{tikzpicture}
        \node[vertex] (B) at (0,0) {};
        \node[vertex] (C) at (1,0) {};
        \draw[edge] (B) to[loop left, min distance=1cm] node[left]{$a$} (B);
        \draw[edge] (B) to node[above]{$b$} (C);
    \end{tikzpicture}
    &%
    \begin{tikzpicture}
        \node[vertex] (B) at (0,0) {};
        \node[vertex] (C) at (1,0) {};
        \draw[edge] (B) to[loop left, min distance=1cm] node[left]{$a$} (B);
        \draw[edge] (C) to node[above]{$b$} (B);
    \end{tikzpicture}%
    \vspace{.05cm}
    \\
    $q_1 \big(\begin{smallmatrix} \mathrm{col}(a) & \mathrm{in} \\ \mathrm{col}(b) & \mathrm{lp} \end{smallmatrix}\big)$%
    &%
    $q_1 \big(\begin{smallmatrix} \mathrm{col}(a) & \mathrm{out} \\ \mathrm{col}(b) & \mathrm{lp} \end{smallmatrix}\big)$%
    &%
    $q_1 \big(\begin{smallmatrix} \mathrm{col}(a) & \mathrm{lp} \\ \mathrm{col}(b) & \mathrm{out} \end{smallmatrix}\big)$%
    &%
    $q_1 \big(\begin{smallmatrix} \mathrm{col}(a) & \mathrm{lp} \\ \mathrm{col}(b) & \mathrm{in} \end{smallmatrix}\big)$%
    \vspace{.2cm}
    \\\hline
    \begin{tikzpicture}
        \node[vertex] (B) at (0,0) {};
        \draw[edge] (B) to[loop left, min distance=1cm] node[left]{$a$} (B);
        \draw[edge] (B) to[loop right, min distance=1cm] node[right]{$b$} (B);
    \end{tikzpicture}
    &%
    \begin{tikzpicture}
        \node[vertex] (A) at (0,0) {};
        \node[vertex] (C) at (1.5,0) {};
        \draw[edge] (A) to[bend right] node[above]{$a$} (C);
        \draw[edge] (A) to[bend left] node[above]{$b$} (C);
    \end{tikzpicture}
    &%
    \begin{tikzpicture}
        \node[vertex] (A) at (0,0) {};
        \node[vertex] (C) at (1.5,0) {};
        \draw[edge] (A) to[bend left] node[above]{$a$} (C);
        \draw[edge] (C) to[bend left] node[above]{$b$} (A);
    \end{tikzpicture}
    &%
    \vspace{.2cm}
    \\
    $q_1 \big(\begin{smallmatrix} \mathrm{col}(a) & \mathrm{lp} \\ \mathrm{col}(b) & \mathrm{lp} \end{smallmatrix}\big)$%
    &%
    $q_1 \big(\begin{smallmatrix} \mathrm{col}(a) & \mathrm{db}^+ \\ \mathrm{col}(b) & \mathrm{db}^+ \end{smallmatrix}\big)$%
    &%
    $q_1 \big(\begin{smallmatrix} \mathrm{col}(a) & \mathrm{db}^+ \\ \mathrm{col}(b) & \mathrm{db}^- \end{smallmatrix}\big)$%
    &%
    \vspace{.2cm}
\end{tabularx}
\caption{Type 1 transitions and the types of adjacency.}
\label{tab:transition:type:1}
\end{table}

\begin{sidewaystable}
\renewcommand{\arraystretch}{2}
\centering
\scriptsize
\begin{tabular}{*{20}{|c}|}
    \cline{1-20}
    \multicolumn{3}{|c|}{} & $\gamma$ &%
    \multicolumn{3}{c|}{$\mathrm{in}$}&%
    \multicolumn{3}{c|}{$\mathrm{out}$}&%
    \multicolumn{3}{c|}{$\mathrm{lp}$}&%
    \multicolumn{6}{c|}{$\mathrm{db}^+$}&%
    \multicolumn{1}{c|}{$\mathrm{db}^-$}\\
    \cline{4-20}
    \multicolumn{3}{|c|}{} & $\iota(a)$ &%
    $\tau_i$ & & $\tau_i$ &%
    $\iota_i$ & & $\iota_i$ &%
    $\lambda_i$ & & $\lambda_i$ &%
    $\tau_i$ & & $\tau_i$ & $\iota_i$ & & $\iota_i$ & any%
    \\
    \cline{4-20}
    \multicolumn{3}{|c|}{} & $\tau(a)$ &%
    & $\tau_i$ & $\tau_i$ &%
    & $\iota_i$ & $\iota_i$ &%
    & $\lambda_i$ & $\lambda_i$ &%
    & $\tau_i$ & $\tau_i$ & & $\iota_i$ & $\iota_i$ & any%
    \\
    \cline{1-20}
    $\delta$ & $\iota(b)$ & $\tau(b)$ & &%
    \multicolumn{3}{c|}{} & \multicolumn{3}{c|}{} & \multicolumn{3}{c|}{} & \multicolumn{6}{c|}{} &\\
    \cline{1-20}
    %
    \multirow{3}{*}{$\mathrm{in}$} &
    $\tau_j$ & & &%
    $\begin{smallmatrix} \mathrm{out} \\ \mathrm{out} \end{smallmatrix}$ & $\begin{smallmatrix} \mathrm{in} \\ \mathrm{out} \end{smallmatrix}$ & $\begin{smallmatrix} \mathrm{lp} \\ \mathrm{out} \end{smallmatrix}$ &%
    $\begin{smallmatrix} \mathrm{out} \\ \mathrm{out} \end{smallmatrix}$ & $\begin{smallmatrix} \mathrm{in} \\ \mathrm{out} \end{smallmatrix}$ & $\begin{smallmatrix} \mathrm{lp} \\ \mathrm{out} \end{smallmatrix}$ &%
    $\begin{smallmatrix} \mathrm{out} \\ \mathrm{out} \end{smallmatrix}$ & $\begin{smallmatrix} \mathrm{in} \\ \mathrm{out} \end{smallmatrix}$ & $\begin{smallmatrix} \mathrm{lp} \\ \mathrm{out} \end{smallmatrix}$ &%
    & & & & & &%
    \multirow{21}{*}{\rotatebox[origin=c]{270}{empty}}\\
    \cline{2-3}\cline{5-19}
    %
    & & $\tau_j$ & &%
    $\begin{smallmatrix} \mathrm{out} \\ \mathrm{in} \end{smallmatrix}$ & $\begin{smallmatrix} \mathrm{in} \\ \mathrm{in} \end{smallmatrix}$ & $\begin{smallmatrix} \mathrm{lp} \\ \mathrm{in} \end{smallmatrix}$ &%
    $\begin{smallmatrix} \mathrm{out} \\ \mathrm{in} \end{smallmatrix}$ & $\begin{smallmatrix} \mathrm{in} \\ \mathrm{in} \end{smallmatrix}$ & $\begin{smallmatrix} \mathrm{lp} \\ \mathrm{in} \end{smallmatrix}$ &%
    $\begin{smallmatrix} \mathrm{out} \\ \mathrm{in} \end{smallmatrix}$ & $\begin{smallmatrix} \mathrm{in} \\ \mathrm{in} \end{smallmatrix}$ & $\begin{smallmatrix} \mathrm{lp} \\ \mathrm{in} \end{smallmatrix}$ &%
    & & & & & &%
    \\
    \cline{2-3}\cline{5-19}
    %
    & $\tau_j$ & $\tau_j$ & &%
    $\begin{smallmatrix} \mathrm{out} \\ \mathrm{lp} \end{smallmatrix}$ & $\begin{smallmatrix} \mathrm{in} \\ \mathrm{lp} \end{smallmatrix}$ & $\begin{smallmatrix} \mathrm{lp} \\ \mathrm{lp} \end{smallmatrix}$ &%
    $\begin{smallmatrix} \mathrm{out} \\ \mathrm{lp} \end{smallmatrix}$ & $\begin{smallmatrix} \mathrm{in} \\ \mathrm{lp} \end{smallmatrix}$ & $\begin{smallmatrix} \mathrm{lp} \\ \mathrm{lp} \end{smallmatrix}$ &%
    $\begin{smallmatrix} \mathrm{out} \\ \mathrm{lp} \end{smallmatrix}$ & $\begin{smallmatrix} \mathrm{in} \\ \mathrm{lp} \end{smallmatrix}$ & $\begin{smallmatrix} \mathrm{lp} \\ \mathrm{lp} \end{smallmatrix}$ &%
    & & & & & &%
    \\
    \cline{1-19}
    %
    \multirow{3}{*}{$\mathrm{out}$} &
    $\iota_j$ & & &%
    $\begin{smallmatrix} \mathrm{out} \\ \mathrm{out} \end{smallmatrix}$ & $\begin{smallmatrix} \mathrm{in} \\ \mathrm{out} \end{smallmatrix}$ & $\begin{smallmatrix} \mathrm{lp} \\ \mathrm{out} \end{smallmatrix}$ &%
    $\begin{smallmatrix} \mathrm{out} \\ \mathrm{out} \end{smallmatrix}$ & $\begin{smallmatrix} \mathrm{in} \\ \mathrm{out} \end{smallmatrix}$ & $\begin{smallmatrix} \mathrm{lp} \\ \mathrm{out} \end{smallmatrix}$ &%
    $\begin{smallmatrix} \mathrm{out} \\ \mathrm{out} \end{smallmatrix}$ & $\begin{smallmatrix} \mathrm{in} \\ \mathrm{out} \end{smallmatrix}$ & $\begin{smallmatrix} \mathrm{lp} \\ \mathrm{out} \end{smallmatrix}$ &%
    & & & & & &%
    \\
    \cline{2-3}\cline{5-19}
    %
    & & $\iota_j$ & &%
    $\begin{smallmatrix} \mathrm{out} \\ \mathrm{in} \end{smallmatrix}$ & $\begin{smallmatrix} \mathrm{in} \\ \mathrm{in} \end{smallmatrix}$ & $\begin{smallmatrix} \mathrm{lp} \\ \mathrm{in} \end{smallmatrix}$ &%
    $\begin{smallmatrix} \mathrm{out} \\ \mathrm{in} \end{smallmatrix}$ & $\begin{smallmatrix} \mathrm{in} \\ \mathrm{in} \end{smallmatrix}$ & $\begin{smallmatrix} \mathrm{lp} \\ \mathrm{in} \end{smallmatrix}$ &%
    $\begin{smallmatrix} \mathrm{out} \\ \mathrm{in} \end{smallmatrix}$ & $\begin{smallmatrix} \mathrm{in} \\ \mathrm{in} \end{smallmatrix}$ & $\begin{smallmatrix} \mathrm{lp} \\ \mathrm{in} \end{smallmatrix}$ &%
    & & & & & &%
    \\
    \cline{2-3}\cline{5-19}
    %
    & $\iota_j$ & $\iota_j$ & &%
    $\begin{smallmatrix} \mathrm{out} \\ \mathrm{lp} \end{smallmatrix}$ & $\begin{smallmatrix} \mathrm{in} \\ \mathrm{lp} \end{smallmatrix}$ & $\begin{smallmatrix} \mathrm{lp} \\ \mathrm{lp} \end{smallmatrix}$ &%
    $\begin{smallmatrix} \mathrm{out} \\ \mathrm{lp} \end{smallmatrix}$ & $\begin{smallmatrix} \mathrm{in} \\ \mathrm{lp} \end{smallmatrix}$ & $\begin{smallmatrix} \mathrm{lp} \\ \mathrm{lp} \end{smallmatrix}$ &%
    $\begin{smallmatrix} \mathrm{out} \\ \mathrm{lp} \end{smallmatrix}$ & $\begin{smallmatrix} \mathrm{in} \\ \mathrm{lp} \end{smallmatrix}$ & $\begin{smallmatrix} \mathrm{lp} \\ \mathrm{lp} \end{smallmatrix}$ &%
    & & & & & &%
    \\
    \cline{1-19}
    %
    \multirow{3}{*}{$\mathrm{lp}$} &
    $\lambda_j$ & & &%
    $\begin{smallmatrix} \mathrm{out} \\ \mathrm{out} \end{smallmatrix}$ & $\begin{smallmatrix} \mathrm{in} \\ \mathrm{out} \end{smallmatrix}$ & $\begin{smallmatrix} \mathrm{lp} \\ \mathrm{out} \end{smallmatrix}$ &%
    $\begin{smallmatrix} \mathrm{out} \\ \mathrm{out} \end{smallmatrix}$ & $\begin{smallmatrix} \mathrm{in} \\ \mathrm{out} \end{smallmatrix}$ & $\begin{smallmatrix} \mathrm{lp} \\ \mathrm{out} \end{smallmatrix}$ &%
    $\begin{smallmatrix} \mathrm{out} \\ \mathrm{out} \end{smallmatrix}$ & $\begin{smallmatrix} \mathrm{in} \\ \mathrm{out} \end{smallmatrix}$ & $\begin{smallmatrix} \mathrm{lp} \\ \mathrm{out} \end{smallmatrix}$ &%
    & & & & & &%
    \\
    \cline{2-3}\cline{5-19}
    %
    & & $\lambda_j$ & &%
    $\begin{smallmatrix} \mathrm{out} \\ \mathrm{in} \end{smallmatrix}$ & $\begin{smallmatrix} \mathrm{in} \\ \mathrm{in} \end{smallmatrix}$ & $\begin{smallmatrix} \mathrm{lp} \\ \mathrm{in} \end{smallmatrix}$ &%
    $\begin{smallmatrix} \mathrm{out} \\ \mathrm{in} \end{smallmatrix}$ & $\begin{smallmatrix} \mathrm{in} \\ \mathrm{in} \end{smallmatrix}$ & $\begin{smallmatrix} \mathrm{lp} \\ \mathrm{in} \end{smallmatrix}$ &%
    $\begin{smallmatrix} \mathrm{out} \\ \mathrm{in} \end{smallmatrix}$ & $\begin{smallmatrix} \mathrm{in} \\ \mathrm{in} \end{smallmatrix}$ & $\begin{smallmatrix} \mathrm{lp} \\ \mathrm{in} \end{smallmatrix}$ &%
    & & & & & &%
    \\
    \cline{2-3}\cline{5-19}
    %
    & $\lambda_j$ & $\lambda_j$ & &%
    $\begin{smallmatrix} \mathrm{out} \\ \mathrm{lp} \end{smallmatrix}$ & $\begin{smallmatrix} \mathrm{in} \\ \mathrm{lp} \end{smallmatrix}$ & $\begin{smallmatrix} \mathrm{lp} \\ \mathrm{lp} \end{smallmatrix}$ &%
    $\begin{smallmatrix} \mathrm{out} \\ \mathrm{lp} \end{smallmatrix}$ & $\begin{smallmatrix} \mathrm{in} \\ \mathrm{lp} \end{smallmatrix}$ & $\begin{smallmatrix} \mathrm{lp} \\ \mathrm{lp} \end{smallmatrix}$ &%
    $\begin{smallmatrix} \mathrm{out} \\ \mathrm{lp} \end{smallmatrix}$ & $\begin{smallmatrix} \mathrm{in} \\ \mathrm{lp} \end{smallmatrix}$ & $\begin{smallmatrix} \mathrm{lp} \\ \mathrm{lp} \end{smallmatrix}$ &%
    & & & & & &%
    \\
    \cline{1-19}
    %
    \multirow{6}{*}{$\mathrm{db}^+$} &
    $\tau_j$ & & &%
    & & &%
    & & &%
    & & &%
    $\begin{smallmatrix} \mathrm{out} \\ \mathrm{out} \end{smallmatrix}$ & $\begin{smallmatrix} \mathrm{in} \\ \mathrm{out} \end{smallmatrix}$ & $\begin{smallmatrix} \mathrm{lp} \\ \mathrm{out} \end{smallmatrix}$ & & & &%
    \\
    \cline{2-3}\cline{5-19}
    %
    & & $\tau_j$ & &%
    & & &%
    & & &%
    & & &%
    $\begin{smallmatrix} \mathrm{out} \\ \mathrm{in} \end{smallmatrix}$ & $\begin{smallmatrix} \mathrm{in} \\ \mathrm{in} \end{smallmatrix}$ & $\begin{smallmatrix} \mathrm{lp} \\ \mathrm{in} \end{smallmatrix}$ & & & &%
    \\
    \cline{2-3}\cline{5-19}
    %
    & $\tau_j$ & $\tau_j$ & &%
    & & &%
    & & &%
    & & &%
    $\begin{smallmatrix} \mathrm{out} \\ \mathrm{lp} \end{smallmatrix}$ & $\begin{smallmatrix} \mathrm{in} \\ \mathrm{lp} \end{smallmatrix}$ & $\begin{smallmatrix} \mathrm{lp} \\ \mathrm{lp} \end{smallmatrix}$ & & & &%
    \\
    \cline{2-3}\cline{5-19}
    %
    & $\iota_j$ & & &%
    & & &%
    & & &%
    & & &%
    & & & $\begin{smallmatrix} \mathrm{out} \\ \mathrm{out} \end{smallmatrix}$ & $\begin{smallmatrix} \mathrm{in} \\ \mathrm{out} \end{smallmatrix}$ & $\begin{smallmatrix} \mathrm{lp} \\ \mathrm{out} \end{smallmatrix}$ &%
    \\
    \cline{2-3}\cline{5-19}
    %
    & & $\iota_j$ & &%
    & & &%
    & & &%
    & & &%
    & & & $\begin{smallmatrix} \mathrm{out} \\ \mathrm{in} \end{smallmatrix}$ & $\begin{smallmatrix} \mathrm{in} \\ \mathrm{in} \end{smallmatrix}$ & $\begin{smallmatrix} \mathrm{lp} \\ \mathrm{in} \end{smallmatrix}$ &%
    \\
    \cline{2-3}\cline{5-19}
    %
    & $\iota_j$ & $\iota_j$ & &%
    & & &%
    & & &%
    & & &%
    & & & $\begin{smallmatrix} \mathrm{out} \\ \mathrm{lp} \end{smallmatrix}$ & $\begin{smallmatrix} \mathrm{in} \\ \mathrm{lp} \end{smallmatrix}$ & $\begin{smallmatrix} \mathrm{lp} \\ \mathrm{lp} \end{smallmatrix}$ &%
    \\
    \cline{1-19}
    %
    \multirow{6}{*}{$\mathrm{db}^-$} &
    $\tau_j$ & & &%
    & & &%
    & & &%
    & & &%
    & & & $\begin{smallmatrix} \mathrm{out} \\ \mathrm{out} \end{smallmatrix}$ & $\begin{smallmatrix} \mathrm{in} \\ \mathrm{out} \end{smallmatrix}$ & $\begin{smallmatrix} \mathrm{lp} \\ \mathrm{out} \end{smallmatrix}$ &%
    \\
    \cline{2-3}\cline{5-19}
    %
    & & $\tau_j$ & &%
    & & &%
    & & &%
    & & &%
    & & & $\begin{smallmatrix} \mathrm{out} \\ \mathrm{in} \end{smallmatrix}$ & $\begin{smallmatrix} \mathrm{in} \\ \mathrm{in} \end{smallmatrix}$ & $\begin{smallmatrix} \mathrm{lp} \\ \mathrm{in} \end{smallmatrix}$ &%
    \\
    \cline{2-3}\cline{5-19}
    %
    & $\tau_j$ & $\tau_j$ & &%
    & & &%
    & & &%
    & & &%
    & & & $\begin{smallmatrix} \mathrm{out} \\ \mathrm{lp} \end{smallmatrix}$ & $\begin{smallmatrix} \mathrm{in} \\ \mathrm{lp} \end{smallmatrix}$ & $\begin{smallmatrix} \mathrm{lp} \\ \mathrm{lp} \end{smallmatrix}$ &%
    \\
    \cline{2-3}\cline{5-19}
    %
    & $\iota_j$ & & &%
    & & &%
    & & &%
    & & &%
    $\begin{smallmatrix} \mathrm{out} \\ \mathrm{out} \end{smallmatrix}$ & $\begin{smallmatrix} \mathrm{in} \\ \mathrm{out} \end{smallmatrix}$ & $\begin{smallmatrix} \mathrm{lp} \\ \mathrm{out} \end{smallmatrix}$ & & & &%
    \\
    \cline{2-3}\cline{5-19}
    %
    & & $\iota_j$ & &%
    & & &%
    & & &%
    & & &%
    $\begin{smallmatrix} \mathrm{out} \\ \mathrm{in} \end{smallmatrix}$ & $\begin{smallmatrix} \mathrm{in} \\ \mathrm{in} \end{smallmatrix}$ & $\begin{smallmatrix} \mathrm{lp} \\ \mathrm{in} \end{smallmatrix}$ & & & &%
    \\
    \cline{2-3}\cline{5-19}
    %
    & $\iota_j$ & $\iota_j$ & &%
    & & &%
    & & &%
    & & &%
    $\begin{smallmatrix} \mathrm{out} \\ \mathrm{lp} \end{smallmatrix}$ & $\begin{smallmatrix} \mathrm{in} \\ \mathrm{lp} \end{smallmatrix}$ & $\begin{smallmatrix} \mathrm{lp} \\ \mathrm{lp} \end{smallmatrix}$ & & & &%
    \\
    \cline{1-20}
\end{tabular}
\caption{Type 2 transitions. Each entry $\begin{smallmatrix} \alpha \\ \beta \end{smallmatrix}$ represents a transition to the state $q_1 \big( \begin{smallmatrix} \mathrm{col}(a) & \alpha \\ \mathrm{col}(b) & \beta \end{smallmatrix} \big)$. An empty entry at $\big( \begin{smallmatrix} \gamma \\ \delta \end{smallmatrix} \big)$ means that the state $q_1 \big( \begin{smallmatrix} i & \gamma \\ j & \delta \end{smallmatrix} \big)$ does not exist.}
\label{tab:transition:type:2}
\end{sidewaystable}

For example, \cref{fig:F:gluing:automaton} depicts the gluing automaton $\mathrm{Gl}_\mathcal{I}$ for the interval replacement system from \cref{fig:replacement:interval} (which is a classical example) and \cref{fig:Dendrite:gluing:automaton} depicts the gluing automaton $\mathrm{Gl}_{\mathcal{D}_n}$ for the Dendrite replacement systems from \cref{fig:replacement:dendrite}.
Finally, \cref{tab:Basilica:gluing:automaton} represents the gluing automaton of the Basilica replacement system:
the automaton $\mathrm{Gl}_\mathcal{B}$ features a transition $q^{(1)} \xrightarrow{(a,b)} q^{(2)}$ if and only if \cref{tab:Basilica:gluing:automaton} features a row that reads $q^{(1)} \mid (a,b) \mid q^{(2)}$.

For example, to deduce the transition $q_1 \big(\begin{smallmatrix} 1 & \mathrm{in} \\ 1 & \mathrm{out} \end{smallmatrix}\big) \xrightarrow{(1,0)} q_1 \big(\begin{smallmatrix} 1 & \mathrm{in} \\ 1 & \mathrm{out} \end{smallmatrix}\big)$ in \cref{fig:F:gluing:automaton} from \cref{tab:transition:type:2} one needs to observe that $\gamma=\mathrm{in}$ and $\tau(1)$ is the terminal vertex $\tau_1$ of the replacement graph, so one needs to consider the second column, and that $\delta=\mathrm{out}$ and $\iota(0)$ is the initial vertex $\iota_1$ of the replacement graph, so one needs to consider the fourth row.
This entry of the table reads $\begin{smallmatrix} \mathrm{in} \\ \mathrm{out} \end{smallmatrix}$.

\begin{figure}
\centering
\begin{tikzpicture}
    \node[state] (start) at (-2.5,0) {$q_0(0)$};
    \node[state] (i) at (0,0) {$q_0(1)$};
    \node[state] (in-out) at (3,1) {$q_1 \big(\begin{smallmatrix} 1 & \mathrm{in} \\ 1 & \mathrm{out} \end{smallmatrix}\big)$};
    \node[state] (out-in) at (3,-1) {$q_1 \big(\begin{smallmatrix} 1 & \mathrm{out} \\ 1 & \mathrm{in} \end{smallmatrix}\big)$};
    \draw[edge] (-3.5,0) to (start);
    \draw[edge] (start) to node[above]{$(S,S)$} (i);
    \draw[edge] (i) to[loop above, min distance=1cm] node[above]{$(0,0)$} (i);
    \draw[edge] (i) to[loop below, min distance=1cm] node[below]{$(1,1)$} (i);
    \draw[edge] (i) to node[above]{$(0,1)$} (in-out);
    \draw[edge] (i) to node[below]{$(1,0)$} (out-in);
    \draw[edge] (in-out) to[loop right, min distance=1cm] node[right]{$(1,0)$} (in-out);
    \draw[edge] (out-in) to[loop right, min distance=1cm] node[right]{$(0,1)$} (out-in);
\end{tikzpicture}
\caption{The gluing automaton $\mathrm{Gl}_\mathcal{I}$ for the interval replacement system $\mathcal{I}$.}
\label{fig:F:gluing:automaton}
\end{figure}
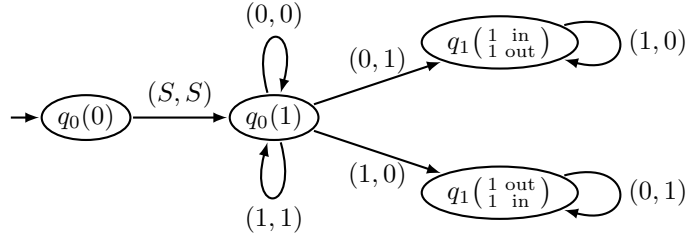

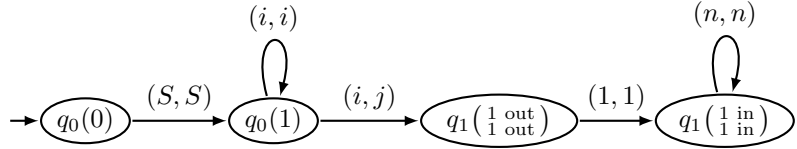
\begin{figure}
\centering
\begin{tikzpicture}
    \node[state] (start) at (-2.5,0) {$q_0(0)$};
    \node[state] (i) at (0,0) {$q_0(1)$};
    \node[state] (out-out) at (3,0) {$q_1 \big(\begin{smallmatrix} 1 & \mathrm{out} \\ 1 & \mathrm{out} \end{smallmatrix}\big)$};
    \node[state] (in-in) at (6,0) {$q_1 \big(\begin{smallmatrix} 1 & \mathrm{in} \\ 1 & \mathrm{in} \end{smallmatrix}\big)$};
    \draw[edge] (-3.5,0) to (start);
    \draw[edge] (start) to node[above]{$(S,S)$} (i);
    \draw[edge] (i) to[loop above, min distance=1cm] node[above]{$(i,i)$} (i);
    \draw[edge] (i) to node[above]{$(i,j)$} (out-out);
    \draw[edge] (out-out) to node[above]{$(1,1)$} (in-in);
    \draw[edge] (in-in) to[loop above, min distance=1cm] node[above]{$(n,n)$} (in-in);
\end{tikzpicture}
\caption{The gluing automaton $\mathrm{Gl}_{\mathcal{D}_n}$ for the Dendrite replacement systems. The label $(i,j)$ indicates that there is a transition for all possible pairs $(i,j)$, where $i$ and $j$ range from $1$ to $n$.}
\label{fig:Dendrite:gluing:automaton}
\end{figure}

\begin{table}
\centering
\renewcommand{\arraystretch}{1.333}
\small
\begin{subtable}[T]{0.45\textwidth}
\centering
{\normalsize Types 0 and 1}\\
\vspace{.25cm}
\begin{tabular}{r|c|l}
    \multirow{4}{*}{$q_0(0)$} & $(L,L)$ & $q_0(\textcolor{red}{2})$ \\
    & $(R,R)$ & $q_0(\textcolor{red}{2})$ \\
    & $(L,R)$ & $q_1 \big(\begin{smallmatrix} \textcolor{red}{2} & \mathrm{lp} \\ \textcolor{red}{2} & \mathrm{lp} \end{smallmatrix}\big)$ \\
    & $(R,L)$ & $q_1 \big(\begin{smallmatrix} \textcolor{red}{2} & \mathrm{lp} \\ \textcolor{red}{2} & \mathrm{lp} \end{smallmatrix}\big)$ \\
    \cline{1-3}
    \multirow{9}{*}{$q_0(\textcolor{blue}{1})$} & $(1,1)$ & $q_0(\textcolor{blue}{1})$ \\
    & $(1,2)$ & $q_1 \big(\begin{smallmatrix} \textcolor{blue}{1} & \mathrm{in} \\ \textcolor{red}{2} & \mathrm{lp} \end{smallmatrix}\big)$ \\
    & $(1,3)$ & $q_1 \big(\begin{smallmatrix} \textcolor{blue}{1} & \mathrm{in} \\ \textcolor{blue}{1} & \mathrm{out} \end{smallmatrix}\big)$ \\
    & $(2,1)$ & $q_1 \big(\begin{smallmatrix} \textcolor{red}{2} & \mathrm{lp} \\ \textcolor{blue}{1} & \mathrm{in} \end{smallmatrix}\big)$ \\
    & $(2,2)$ & $q_0(\textcolor{red}{2})$ \\
    & $(2,3)$ & $q_1 \big(\begin{smallmatrix} \textcolor{red}{2} & \mathrm{lp} \\ \textcolor{blue}{1} & \mathrm{out} \end{smallmatrix}\big)$ \\
    & $(3,1)$ & $q_1 \big(\begin{smallmatrix} \textcolor{blue}{1} & \mathrm{out} \\ \textcolor{blue}{1} & \mathrm{in} \end{smallmatrix}\big)$ \\
    & $(3,2)$ & $q_1 \big(\begin{smallmatrix} \textcolor{blue}{1} & \mathrm{out} \\ \textcolor{red}{2} & \mathrm{lp} \end{smallmatrix}\big)$ \\
    & $(3,3)$ & $q_0(\textcolor{blue}{1})$ \\
    \cline{1-3}
    \multirow{9}{*}{$q_0(\textcolor{red}{2})$} & $(4,4)$ & $q_0(\textcolor{blue}{1})$ \\
    & $(4,5)$ & $q_1 \big(\begin{smallmatrix} \textcolor{blue}{1} & \mathrm{in} \\ \textcolor{red}{2} & \mathrm{lp} \end{smallmatrix}\big)$ \\
    & $(4,6)$ & $q_1 \big(\begin{smallmatrix} \textcolor{blue}{1} & \mathrm{db}^+ \\ \textcolor{blue}{1} & \mathrm{db}^- \end{smallmatrix}\big)$ \\
    & $(5,4)$ & $q_1 \big(\begin{smallmatrix} \textcolor{red}{2} & \mathrm{lp} \\ \textcolor{blue}{1} & \mathrm{in} \end{smallmatrix}\big)$ \\
    & $(5,5)$ & $q_0(\textcolor{red}{2})$ \\
    & $(5,6)$ & $q_1 \big(\begin{smallmatrix} \textcolor{red}{2} & \mathrm{lp} \\ \textcolor{blue}{1} & \mathrm{out} \end{smallmatrix}\big)$ \\
    & $(6,4)$ & $q_1 \big(\begin{smallmatrix} \textcolor{blue}{1} & \mathrm{db}^+ \\ \textcolor{blue}{1} & \mathrm{db}^- \end{smallmatrix}\big)$ \\
    & $(6,5)$ & $q_1 \big(\begin{smallmatrix} \textcolor{blue}{1} & \mathrm{out} \\ \textcolor{red}{2} & \mathrm{lp} \end{smallmatrix}\big)$ \\
    & $(6,6)$ & $q_0(\textcolor{blue}{1})$ \\
\end{tabular}
\end{subtable}
\begin{subtable}[T]{0.45\textwidth}
\centering
{\normalsize Type 2}\\
\vspace{.25cm}
\begin{tabular}{r|c|l}
    \multirow{4}{*}{$q_1 \big(\begin{smallmatrix} \textcolor{red}{2} & \mathrm{lp} \\ \textcolor{red}{2} & \mathrm{lp} \end{smallmatrix}\big)$} & $(4,4)$ & $q_1 \big(\begin{smallmatrix} \textcolor{blue}{1} & \mathrm{out} \\ \textcolor{blue}{1} & \mathrm{out} \end{smallmatrix}\big)$ \\
    & $(4,6)$ & $q_1 \big(\begin{smallmatrix} \textcolor{blue}{1} & \mathrm{out} \\ \textcolor{blue}{1} & \mathrm{in} \end{smallmatrix}\big)$ \\
    & $(6,4)$ & $q_1 \big(\begin{smallmatrix} \textcolor{blue}{1} & \mathrm{in} \\ \textcolor{blue}{1} & \mathrm{out} \end{smallmatrix}\big)$ \\
    & $(6,6)$ & $q_1 \big(\begin{smallmatrix} \textcolor{blue}{1} & \mathrm{in} \\ \textcolor{blue}{1} & \mathrm{in} \end{smallmatrix}\big)$ \\
    \cline{1-3}
    \multirow{2}{*}{$q_1 \big(\begin{smallmatrix} \textcolor{blue}{1} & \mathrm{in} \\ \textcolor{red}{2} & \mathrm{lp} \end{smallmatrix}\big)$} & $(3,4)$ & $q_1 \big(\begin{smallmatrix} \textcolor{blue}{1} & \mathrm{in} \\ \textcolor{blue}{1} & \mathrm{out} \end{smallmatrix}\big)$ \\
    & $(3,6)$ & $q_1 \big(\begin{smallmatrix} \textcolor{blue}{1} & \mathrm{in} \\ \textcolor{blue}{1} & \mathrm{in} \end{smallmatrix}\big)$ \\
    \cline{1-3}
    \multirow{1}{*}{$q_1 \big(\begin{smallmatrix} \textcolor{blue}{1} & \mathrm{in} \\ \textcolor{blue}{1} & \mathrm{out} \end{smallmatrix}\big)$} & $(3,1)$ & $q_1 \big(\begin{smallmatrix} \textcolor{blue}{1} & \mathrm{in} \\ \textcolor{blue}{1} & \mathrm{out} \end{smallmatrix}\big)$ \\
    \cline{1-3}
    \multirow{2}{*}{$q_1 \big(\begin{smallmatrix} \textcolor{red}{2} & \mathrm{lp} \\ \textcolor{blue}{1} & \mathrm{in} \end{smallmatrix}\big)$} & $(4,3)$ & $q_1 \big(\begin{smallmatrix} \textcolor{blue}{1} & \mathrm{out} \\ \textcolor{blue}{1} & \mathrm{in} \end{smallmatrix}\big)$ \\
    & $(6,3)$ & $q_1 \big(\begin{smallmatrix} \textcolor{blue}{1} & \mathrm{in} \\ \textcolor{blue}{1} & \mathrm{in} \end{smallmatrix}\big)$ \\
    \cline{1-3}
    \multirow{2}{*}{$q_1 \big(\begin{smallmatrix} \textcolor{red}{2} & \mathrm{lp} \\ \textcolor{blue}{1} & \mathrm{out} \end{smallmatrix}\big)$} & $(4,1)$ & $q_1 \big(\begin{smallmatrix} \textcolor{blue}{1} & \mathrm{out} \\ \textcolor{blue}{1} & \mathrm{out} \end{smallmatrix}\big)$ \\
    & $(6,1)$ & $q_1 \big(\begin{smallmatrix} \textcolor{blue}{1} & \mathrm{in} \\ \textcolor{blue}{1} & \mathrm{out} \end{smallmatrix}\big)$ \\
    \cline{1-3}
    \multirow{1}{*}{$q_1 \big(\begin{smallmatrix} \textcolor{blue}{1} & \mathrm{out} \\ \textcolor{blue}{1} & \mathrm{in} \end{smallmatrix}\big)$} & $(1,3)$ & $q_1 \big(\begin{smallmatrix} \textcolor{blue}{1} & \mathrm{out} \\ \textcolor{blue}{1} & \mathrm{in} \end{smallmatrix}\big)$ \\
    \cline{1-3}
    \multirow{2}{*}{$q_1 \big(\begin{smallmatrix} \textcolor{blue}{1} & \mathrm{out} \\ \textcolor{red}{2} & \mathrm{lp} \end{smallmatrix}\big)$} & $(1,4)$ & $q_1 \big(\begin{smallmatrix} \textcolor{blue}{1} & \mathrm{out} \\ \textcolor{blue}{1} & \mathrm{out} \end{smallmatrix}\big)$ \\
    & $(1,6)$ & $q_1 \big(\begin{smallmatrix} \textcolor{blue}{1} & \mathrm{out} \\ \textcolor{blue}{1} & \mathrm{in} \end{smallmatrix}\big)$ \\
    \cline{1-3}
    \multirow{1}{*}{$q_1 \big(\begin{smallmatrix} \textcolor{blue}{1} & \mathrm{out} \\ \textcolor{blue}{1} & \mathrm{out} \end{smallmatrix}\big)$} & $(1,1)$ & $q_1 \big(\begin{smallmatrix} \textcolor{blue}{1} & \mathrm{out} \\ \textcolor{blue}{1} & \mathrm{out} \end{smallmatrix}\big)$ \\
    \cline{1-3}
    \multirow{1}{*}{$q_1 \big(\begin{smallmatrix} \textcolor{blue}{1} & \mathrm{in} \\ \textcolor{blue}{1} & \mathrm{in} \end{smallmatrix}\big)$} & $(3,3)$ & $q_1 \big(\begin{smallmatrix} \textcolor{blue}{1} & \mathrm{in} \\ \textcolor{blue}{1} & \mathrm{in} \end{smallmatrix}\big)$ \\
    \cline{1-3}
    \multirow{2}{*}{$q_1 \big(\begin{smallmatrix} \textcolor{blue}{1} & \mathrm{db}^+ \\ \textcolor{blue}{1} & \mathrm{db}^- \end{smallmatrix}\big)$} & $(1,3)$ & $q_1 \big(\begin{smallmatrix} \textcolor{blue}{1} & \mathrm{out} \\ \textcolor{blue}{1} & \mathrm{in} \end{smallmatrix}\big)$ \\
    & $(3,1)$ & $q_1 \big(\begin{smallmatrix} \textcolor{blue}{1} & \mathrm{in} \\ \textcolor{blue}{1} & \mathrm{out} \end{smallmatrix}\big)$ \\
\end{tabular}
\end{subtable}
\caption{The gluing automaton $\mathrm{Gl}_\mathcal{B}$ for the Basilica.}
\label{tab:Basilica:gluing:automaton}
\end{table}

\section{Rationality of the Gluing Relation}
\label{sec:gl:aut}

The definition of the gluing automaton $\mathrm{Gl}_\mathcal{R}$ might look complicated, but each of its transition has a simple meaning.
This Subsection features a series of Lemmas that describe how the gluing automaton works and will be useful when proving the rationality of the gluing relation.

\begin{lemma}[Type 0 transitions generate the color graph]
\label{lem:0:trans}
The full sub-automaton $\mathrm{T}_\mathcal{R} = (\Sigma^2, Q_0, \rightarrow, q_0)$ induced by the states $Q_0$ is naturally isomorphic to the color graph.
More precisely, a finite word $(x_0, y_0) \dots (x_m, y_m)$ in the alphabet $\Sigma^2$ is recognized by $\mathrm{T}_\mathcal{R}$ if and only if $x_i = y_i$ and the word $x_0 \dots x_m$ belongs to $E_\mathcal{R}$.
\end{lemma}

\begin{proof}
Note that $\mathrm{Gl}_\mathcal{R}$ features no transition from $Q_1$ to $Q_0$ and the transitions from $Q_0$ to itself are precisely the Type 0 transitions.
The map $\mathrm{T}_\mathcal{R} \rightarrow A_\mathcal{R}$ that is the identity on $Q_0$ and sends the transition $q_0(i) \xrightarrow{(e,e)} q_0 (\mathrm{col}(e))$ to $q_0(i) \xrightarrow{e} q_0 (\mathrm{col}(e))$ is an isomorphism.
\end{proof}

Note that if two words $x_0 \dots x_m$ and $y_0 \dots y_m$ in $E_\mathcal{R}$ are the same except for $x_m \neq y_m$, then they represent adjacent edges in the $m$-th full expansion $E_m$ (\cref{def.full.expansion}) if and only if $x_m$ and $y_m$ are adjacent in $\Gamma_{\mathrm{Col}(x_{m-1})} = \Gamma_{\mathrm{Col}(y_{m-1})}$.
Thus, even if by definition Type 1 transitions depend on edge adjacencies in the base or the replacement graphs, they really determine adjacencies in $E_m$.
This is because identifying the initial and terminal vertices of replacement graphs for colors associated to loops (as explained in \cref{sub.loops}) allows us to see the expanded subgraph as an embedded copy of the replacement graph.

\begin{lemma}[Type 1 transition implies adjacency]
\label{lem:1:trans:imply:adj}
If a sequence $(x_0, y_0) (x_1, y_1) \dots$ of $(\Sigma^2)^\omega$ is recognized by $\mathrm{Gl}_\mathcal{R}$, then one and only one of the following hold.
\begin{enumerate}
    \item The sequence is recognized by the sub-automaton $\mathrm{T}_\mathcal{R}$ and $x_i=y_i$ for all $i$.
    \item The run of $(x_0, y_0) (x_1, y_1) \dots$ in $\mathrm{Gl}_\mathcal{R}$ features exactly one Type 1 transition taking place when processing the first $(x_m, y_m)$ such that $x_m \neq y_m$.
    Moreover, $x_0 \dots x_m$ and $y_0 \dots y_m$ are adjacent edges in $E_m$ and the state of $\mathrm{Gl}_\mathcal{R}$ after processing $(x_m, y_m)$ is $q_1 \big(\begin{smallmatrix} \mathrm{col}(x_m) & \alpha \\ \mathrm{col}(y_m) & \beta \end{smallmatrix}\big)$, which determines the type of adjacency by the entries $\alpha$ and $\beta$, as describe in \cref{tab:transition:type:1}.
\end{enumerate}
\end{lemma}

\begin{proof}
Case (1) is the one discussed in \cref{lem:0:trans}, so let us assume that there is an $m \in \mathbb{N}$ such that $x_m \neq y_m$.
Take $m$ to be the minimum such number.
Then, before the letter $(x_m, y_m)$, only Type 0 transitions have taken place.
Note that $\mathrm{Gl}_\mathcal{R}$ features no transitions from $Q_1$ to $Q_0$.
Thus, since Type 1 transitions are all oriented from $Q_0$ to $Q_1$, each run can feature at most one Type 1 transition.

By \cref{lem:0:trans}, $x_0 \dots x_{m-1} = y_0 \dots y_{m-1}$ belongs to $E_\mathcal{R}$, so it is an edge of $E_{m-1}$.
By the definition of a Type 1 transition, it follows that $x_0 \dots x_m$ and $y_0 \dots y_m$ are adjacent in $E_m$ and that the state of $\mathrm{Gl}_\mathcal{R}$ after processing $(x_m, y_m)$ is $q_1 \big(\begin{smallmatrix} \mathrm{col}(x_m) & \alpha \\ \mathrm{col}(y_m) & \beta \end{smallmatrix}\big)$, where $\alpha$ and $\beta$ describe the type adjacency as expressed in \cref{tab:transition:type:1}.
\end{proof}

\begin{lemma}[Adjacency implies Type 1 transition]
\label{lem:adj:imply:1:trans}
Given two words $x_0 \dots x_m$ and $y_0 \dots y_m$ in $E_\mathcal{R}$ that are the same except for $x_m \neq y_m$, if they represent adjacent edges of $E_m$ then the word $(x_0, y_0) (x_1, y_1) \dots (x_m, y_m)$ is recognized by $\mathrm{Gl}_\mathcal{R}$, and its run consists of $m$ Type 0 transitions followed by a Type 1 transition.
\\
Moreover, the type of adjacency between $x_m$ and $y_m$ in $\Gamma_{\mathrm{Col}(x_{m-1})} = \Gamma_{\mathrm{Col}(y_{m-1})}$ and between $x_0 \dots x_m$ and $y_0 \dots y_m$ in $E_m$ is the same, i.e., the last state of the run $q_1 \big(\begin{smallmatrix} \mathrm{col}(x_m) & \alpha \\ \mathrm{col}(y_m) & \beta \end{smallmatrix}\big)$
determines the edge adjacency of $x_0 \dots x_m$ and $y_0 \dots y_m$ in $E_m$ as described in \cref{tab:transition:type:1}.
\end{lemma}

\begin{proof}
Given two words $x_0 \dots x_m$ and $y_0 \dots y_m$ in $E_\mathcal{R}$ that only differ in $x_m \neq y_m$ and that represent adjacent edges of $E_m$, by \cref{lem:0:trans} the sub-automaton $\mathrm{T}_\mathcal{R}$ recognizes $(x_0, y_0) (x_1, y_1) \ldots (x_{m-1}, y_{m-1})$.
By definition, a Type 1 transition occurs when processing the letter $(x_m, y_m)$, ending in a state $q_1 \big(\begin{smallmatrix} \mathrm{col}(x_m) & \alpha \\ \mathrm{col}(y_m) & \beta \end{smallmatrix}\big)$ that determines the edge adjacency between $x_m$ and $y_m$ in the replacement graph $\Gamma_{\mathrm{col}(x_{m-1})} = \Gamma_{\mathrm{col}(y_{m-1})}$, and thus in $E_m$, as expressed in \cref{tab:transition:type:1}.
\end{proof}

Note that a Type 2 transition is preceded by a (possibly empty) sequence of Type 0 transitions, one Type 1 transition and possibly a sequence of Type 2 transitions.

\begin{lemma}[Type 2 transitions imply adjacency]
\label{lem:2:trans:imply:adj}
Let $(x_0, y_0) (x_1, y_1) \ldots (x_k, y_k)$ be recognized by $\mathrm{Gl}_\mathcal{R}$.
If the transition involving $(x_k, y_k)$ is a Type 2 transition, then $x_0 \dots x_k$ and $y_0 \dots y_k$ are adjacent edges of $E_k$.
\end{lemma}

\begin{proof}
First, suppose that the transition involving $(x_k, y_k)$ is the first Type 2 transition, so the previous transition is of Type 1.
Then, by \cref{lem:1:trans:imply:adj}, the edges $x_1 \dots x_{k-1}$ and $y_1 \dots y_{k-1}$ are adjacent and the state $q_1 \big(\begin{smallmatrix} \mathrm{col}(x_{k-1}) & \alpha \\ \mathrm{col}(y_{k-1}) & \beta \end{smallmatrix}\big)$ after processing $(x_{k-1}, y_{k-1})$ describes the type of adjacency as expressed in \cref{tab:transition:type:1}.
Thoroughly going through \cref{tab:transition:type:2} 
reveals that the column and the row of any non empty entry correspond to a pair of adjacent edges in the graph obtained by gluing the replacement graphs $\Gamma_{\mathrm{col}(x_{k-1})}$ and $\Gamma_{\mathrm{col}(y_{k-1})}$ as dictated by $\alpha$ and $\beta$, which appears as a subgraph of $E_k$ that contains the edges $x_1 \dots x_k$ and $y_1 \dots y_k$.
Moreover, the resulting state $q_1 \big(\begin{smallmatrix} \mathrm{col}(x_k) & \alpha' \\ \mathrm{col}(y_k) & \beta' \end{smallmatrix}\big)$ describes the adjacency between $x_0 \dots x_k$ and $y_0 \dots y_k$.

Now suppose that the transition involving $(x_k, y_k)$ is not the first Type 2 transition.
The argument of the previous paragraph applies as soon as the last state in the run of $(x_1,y_1) \dots (x_{k-1}, y_{k-1})$ describes the adjacency of $(x_{k-1}, y_{k-1})$.
Thus, we can iteratively apply the argument to see that $x_0 \dots x_k$ and $y_0 \dots y_k$ are adjacent and that the last state of the run determines their type of adjacency.
\end{proof}

\begin{lemma}[Adjacency implies a Type 2 transition]
\label{lem:adj:imply:2:trans}
Let $x_0 \dots x_{m+1}$ and $y_0 \dots y_{m+1}$ in $E_\mathcal{R}$ represent adjacent edges of $E_{m+1}$.
Suppose that $x_0 \dots x_m \neq y_0 \dots y_m$ and that $(x_0, y_0) \dots (x_m, y_m)$ is recognized by $\mathrm{Gl}_\mathcal{R}$.
Then $(x_0, y_0) \dots (x_{m+1}, y_{m+1})$ is recognized by $\mathrm{Gl}_\mathcal{R}$ and the last transition is a Type 2 transition ending in $q_1 \big(\begin{smallmatrix} \mathrm{col}(x_{m+1}) & \alpha \\ \mathrm{col}(y_{m+1}) & \beta \end{smallmatrix}\big)$ that describes the adjacency between $x_0 \dots x_{m+1}$ and $y_0 \dots y_{m+1}$ in $E_{m+1}$.
\end{lemma}

\begin{proof}
First, suppose that the state after having processed $(x_0, y_0), \dots (x_m, y_m)$ is $q_1 \big(\begin{smallmatrix} \mathrm{col}(x_m) & \alpha \\ \mathrm{col}(y_m) & \beta \end{smallmatrix}\big)$ and that it describes the type adjacency as expressed in \cref{tab:transition:type:1}.
Then a careful look at \cref{tab:transition:type:2} 
shows that the gluing automaton can process the next letter $(x_{m+1}, y_{m+1})$ with a Type 2 transition that ends in the state $q_1 \big(\begin{smallmatrix} \mathrm{col}(x_{m+1}) & \alpha' \\ \mathrm{col}(y_{m+1}) & \beta' \end{smallmatrix}\big)$ which describes the adjacency between $x_0 \dots x_{m+1}$ and $y_0 \dots y_{m+1}$.

Now, by \cref{lem:adj:imply:1:trans} the common prefix $x_0 \dots x_p = y_0 \dots y_p$ of $x_0 \dots x_m$ and $y_0 \dots y_m$ (where $0 \leq p \leq m-1$) is processed by the sub-automaton $\mathrm{T}_\mathcal{R}$, after which a Type 1 transition brings us to the state $q_1 \big(\begin{smallmatrix} \mathrm{col}(x_{p+1}) & \alpha \\ \mathrm{col}(y_{p+1}) & \beta \end{smallmatrix}\big)$ that describes the adjacency of $x_0 \dots x_{p+1}$ and $y_0 \dots y_{p+1}$ in $E_{p+1}$.
Reasoning as in the previous paragraph shows that this is followed by a Type 2 transition to the state $q_1 \big(\begin{smallmatrix} \mathrm{col}(x_{p+2}) & \alpha' \\ \mathrm{col}(y_{p+2}) & \beta' \end{smallmatrix}\big)$ that describes the adjacency between $x_0 \dots x_{p+2}$ and $y_0 \dots y_{p+2}$.
Since each new transition results in a state that describes edge adjacency, the previous paragraph can be iteratively applied to every subsequent prefix of $x_0 \dots x_{m+1}$ and $y_0 \dots y_{m+1}$, concluding with a Type 2 transition from $(x_0, y_0) \dots (x_m, y_m)$ to $(x_0, y_0) \dots (x_{m+1}, y_{m+1})$ that ends in a state that describes the edge adjacency between $x_0 \dots x_{m+1}$ and $y_0 \dots y_{m+1}$.
\end{proof}

Now let us put the Lemmas together.

\begin{theorem}
\label{thm:rat:gl}
The gluing relation of an expanding replacement system is rational.
\end{theorem}

\begin{proof}
We have to prove that $(x_0, y_0) (x_1, y_1) \ldots$ is recognized by $\mathrm{Gl}_\mathcal{R}$ if and only if the sequences $x_0 x_1 \ldots$ and $y_0 y_1 \ldots$ are equivalent under the gluing relation $\sim$.

\medskip

Suppose $v = x_0 x_1 \dots$ and $w = y_0 y_1 \dots$ are equivalent.
If $v = w$ then $\mathrm{Gl}_\mathcal{R}$ recognizes the sequence $(v,w)$ by \cref{lem:0:trans}, so we are done.
Suppose that instead $v \neq w$ and let $m \geq 0$ be the smallest natural number such that $x_m \neq y_m$.
Since $v \sim w$, by \cref{def:gluing} there exists a unique vertex on which $x_0 \ldots x_k$ and $y_0 \ldots y_k$ are incident as edges of $E_k$, for all $k > m$ (for $k=m$ the edges may be parallel, in which case there may be two such vertices).
In particular, $x_0 \ldots x_m$ and $y_0 \ldots y_m$ are adjacent, so by \cref{lem:adj:imply:1:trans} the prefix $(x_0, y_0) \ldots (x_m, y_m)$ is recognized by $\mathrm{Gl}_\mathcal{R}$ and its run ends with a Type 1 transition to the state $q_1 \big(\begin{smallmatrix} \mathrm{col}(x_m) & \alpha \\ \mathrm{col}(y_m) & \beta \end{smallmatrix}\big)$, which describes the type of adjacency between $x_1 \dots x_m$ and $y_0 \dots y_m$ in $E_m$.

For all of the subsequent prefixes of our sequence $(v,w)$ we are in the case described in \cref{lem:adj:imply:2:trans}, which shows that there is a Type 2 transition from any of those prefixes to the next one, ultimately showing that $(v,w)$ is recognized by $\mathrm{Gl}_\mathcal{R}$.

\medskip

Conversely, suppose that $(v, w) = (x_0, y_0) (x_1, y_1) \ldots$ is recognized by $\mathrm{Gl}_\mathcal{R}$.
We need to show that $v$ and $w$ belong to the symbol space $\Omega_\mathcal{R}$ and that $ v \sim w $.

Consider the morphism $\phi$ from the underlying graph of $\mathrm{Gl}_\mathcal{R}$ to the color graph $A_\mathcal{R}$ defined by mapping the states as follows:
    \[ q_0(i) \text{ is mapped to } q_0(i), \]
    \[ q_1 \big(\begin{smallmatrix} \mathrm{col}(a) & \alpha \\ \mathrm{col}(b) & \beta \end{smallmatrix}\big) \text{ is mapped to } q_0(\mathrm{col}(a)); \]
and by mapping the transitions as follows:
    \[ q_0(i) \xrightarrow{(e,e)} q_0 (\mathrm{col}(e)) \text{ is mapped to } q_0(i) \rightarrow q_0 (\mathrm{col}(e)), \]
    \[ q_0(i) \xrightarrow{(a,b)} q_1 \big(\begin{smallmatrix} \mathrm{col}(a) & \alpha \\ \mathrm{col}(b) & \beta \end{smallmatrix}\big) \text{ is mapped to } q_0(i) \rightarrow q_0 (\mathrm{col}(a)), \]
    \[ q_1 \big(\begin{smallmatrix} i & \gamma \\ j & \delta \end{smallmatrix}\big) \xrightarrow{(a,b)} q_1 \big(\begin{smallmatrix} \mathrm{col}(a) & \alpha \\ \mathrm{col}(b) & \beta \end{smallmatrix}\big) \text{ is mapped to } q_0(i) \rightarrow q_0 (\mathrm{col}(a)). \]
By definition of Type 1 and 2 transitions, in the schemes above $a$ is an edge of the graph $\Gamma_i$.
Thus, $\phi(\mathrm{Gl}_\mathcal{R})$ is a subgraph of the color graph, so $\phi: \mathrm{Gl}_\mathcal{R} \to A_\mathcal{R}$ is a well-defined graph morphism.
In fact, it is easy to see that $\phi(\mathrm{Gl}_\mathcal{R})$ is precisely $A_\mathcal{R}$.
Thus, if $(v,w)$ is recognized by $\mathrm{Gl}_\mathcal{R}$ then $v$ is in the edge shift $\Omega_\mathcal{R}$ of the color graph.
The same holds for the entry $w$, since $\mathrm{Gl}_\mathcal{R}$ is symmetric under switching inputs $(a,b)$ with $(b,a)$ and states $q_i \big(\begin{smallmatrix} i & \gamma \\ j & \delta \end{smallmatrix}\big)$ with $q_i \big(\begin{smallmatrix} j & \delta \\ i & \gamma \end{smallmatrix}\big)$.

We now show that $v \sim w$.
If $v = w$ there is nothing to prove, so let $v \neq w$.
By \cref{lem:1:trans:imply:adj}, the run of $(v, w)$ features finitely many Type 0 transitions followed by a Type 1 transition at $(x_0, y_0), \dots (x_m, y_m)$, for some $m \geq 0$, that brings us to a state of $Q_1$, and $x_0 \dots x_m$ is adjacent to $y_0 \dots y_m$.
This shows that the prefixes of $v$ and $w$ of equal length $k \leq m+1$ represent adjacent edges.
Since the transitions of the gluing automaton that start in $Q_1$ are of Type 2 and end in $Q_1$, every subsequent transition of the run of $(v, w)$ is a Type 2 transition.
By \cref{lem:2:trans:imply:adj}, the prefixes of $v$ and $w$ of equal length $k > m+1$ are adjacent, concluding our proof.
\end{proof}


\section*{Acknowledgements}

The authors are thankful to Francesco Matucci for helpful advice, to Jim Belk and Bradley Forrest for kindly providing the image used in \cref{fig:basilica} and to the anonymous referee for insightful observations that led to clearer exposition.


\printbibliography[heading=bibintoc]

\end{document}